\documentclass[preprint,10pt]{elsarticle}
\usepackage[T1]{fontenc}

\usepackage{graphicx}
\usepackage{thm-restate}
\usepackage{amsmath}
\usepackage{verbatim}
\usepackage{amssymb}
\usepackage{amsthm}

\theoremstyle{definition}
\newtheorem{definition}{Definition}[section]

\theoremstyle{plain}
\newtheorem{prop}{Property}[section]
\newtheorem{thm}{Theorem}
\newtheorem{conj}{Conjecture}
\newtheorem{claim}{Claim}

\newtheorem{question}{Question}

\usepackage{hyperref}

\usepackage{tikz}
\usetikzlibrary{arrows, quotes}
\usetikzlibrary{arrows.meta}
\newcommand{\stild}{\widetilde{\sigma}}

\begin{document}

\begin{frontmatter}

\title{Antimagicness of graphs with a dominating clique}
\author{Grégoire Beaudoire, Cédric Bentz, Christophe Picouleau}

\affiliation{organization={CEDRIC, Conservatoire National des Arts et Métiers},city={Paris}, country={France}}

\begin{abstract}
A graph $G = (V,E)$ is called \emph{antimagic} if there exists a bijective labelling $f : E \rightarrow \{1,2,\ldots,|E|\}$ such that the vertex-sums of labels over edges incident to a given vertex are all distinct. In this paper, we extend the antimagicness results over graphs with a dominating clique. We also introduce an alternative to the usual definition of antimagic graphs, called \emph{$C$-antimagic}, allowing for the labelling to be injective in $\{1,2,\ldots,|E|+C\}$ instead of bijective, and show that almost all graphs with a dominating clique are $3$-antimagic.
\end{abstract}


\end{frontmatter}

\section{Introduction and definitions}

In this paper, we only consider finite, simple, and undirected graphs. We refer to \cite{west} for undefined terminology.

Let $G = (V,E)$ be a graph with $|V| = n$ and $|E| = m$. For each vertex $v \in V$, we will denote by $N_G(v)$ the set of neighbors of $v$ in $G$. When $G$ is obvious from the context, we will simply write $N(v)$. We will write $d_G(v) = |N_G(v)|$, and simply $d(v)$ when $G$ is clear from the context. If $d(v) = n - 1$, $v$ is called a \emph{universal vertex}. If all the vertices of $G$ are universal, then $G$ is a \emph{complete graph}. For any two disjoint subsets $A,B \subset V$, we will denote by $E(A,B) \subset E$ the set of edges with an endpoint in $A$ and the other in $B$.

For any undirected path $v_1,v_2,\ldots,v_k$ in $G$ we will call $v_2,\ldots,v_{k-1}$ \emph{interior vertices} and $v_1$ and $v_k$ the \emph{endpoints} of the path.

The subgraph of $G = (V,E)$ \emph{induced} by a set $S \subseteq V$ is denoted by $G[S]$, and defined as $G[S] = (S,F)$ where $F = \{xy \in E | x,y \in S\}$. If $G$ has an induced complete subgraph $H$, $H$ is called a \emph{clique} in $G$. Additionally, if $H$ is such that every vertex $v \in V(G) \setminus V(H)$ has a neighbor in $H$, $H$ is called a \emph{dominating clique} in $G$. For the sake of simplicity, in the rest of the article, we will use $H$ to denote both the vertex set of a dominating clique and the dominating clique itself. If $G$ has an induced subgraph $H$ such that every vertex in $H$ has degree $0$, $H$ is called an \emph{independent set} in $G$.

Given a graph $G = (V,E)$, let $f : E \rightarrow \{1,2,\ldots,m\}$ be a bijective labelling of the edges of $G$. For each vertex $u \in V$, we will denote by $\sigma(u) = \sum\limits_{v \in V | uv \in E} f(uv)$ the sum of the labels over the edges incident to $u$. If all the values of $\sigma(u)$ are pairwise distinct for all $u \in V$, then $f$ is called an \emph{antimagic labelling} of $G$, and $G$ is said to be \emph{antimagic}.

Antimagic labelling was originally introduced by Hartsfield and Ringel in 1990 \cite{hartsfield1990}, where they introduced the following (still open) conjecture:

\begin{conj}
    \label{conjcentrale}
    Every connected graph other than $K_2$ is antimagic.
\end{conj}

The topic is the focus of a chapter of 12 pages in the dynamic survey on graph labelling, updated yearly by J. Gallian \cite{survey}.

Our paper mainly focuses on graphs with a dominating clique. We survey the existing results over graphs with such a structure; the first result was proved by Barrus in 2010 \cite{barrus2010}:

\begin{thm}\label{barrused}
    Let $G = (V,E)$ be a graph with at least $3$ vertices. If $G$ has a clique $B$ such that, for every vertex $v \in V$, either $N(v) \subset B$ or $B \subset N(v)$, then $G$ is antimagic.
\end{thm}

Notice that the case $|B| = 1$ means that $G$ has a universal vertex, implying $G$ is antimagic \cite{alon2004}.

We give an illustration of the structure of graphs described in Theorem $1$ in Figure \ref{barrusillus}.

\begin{figure}[h]
        \centering
        \begin{tikzpicture}
            \node at (0,-2) {$B$};

            \node at (-4,0) (u1) {$\bullet$};
            \node at (-4,1) (u2) {$\bullet$};
            \node at (-4,-1) (u3) {$\bullet$};

            \node at (-1,1) (v1) {$\bullet$};
            \node at (1,1) (v2) {$\bullet$};
            \node at (1,-1) (v3) {$\bullet$};
            \node at (-1,-1) (v4) {$\bullet$};

            \draw (u1.center) -- (v1.center);
            \draw (u2.center) -- (v1.center);
            \draw (u3.center) -- (v2.center);
            \draw (u3.center) -- (v3.center);

            \draw (v1.center) -- (v2.center) -- (v3.center) -- (v1.center);

            \draw (v1.center) -- (v4.center) -- (v2.center);
            \draw (v4.center) -- (v3.center);

            \node at (4,0.5) (w1) {$\bullet$};
            \node at (4,-0.5) (w2) {$\bullet$};
            \foreach \y in {1,2}
            \foreach \x in {1,2,3,4}
                {
                \draw (v\x.center) -- (w\y.center);
                }

            \node at (-4, -2) {$A$};
            \node at (4,-2) {$C$};
        \end{tikzpicture}
        \caption{Illustration of the structure of graphs described in Theorem \ref{barrused}: $A$ is an independent set made of vertices $u$ such that $N(u) \subset B$, and $C$ is the set of vertices $v$ such that $B \subset N(v)$ (with no particular structure).}
        \label{barrusillus}
    \end{figure}
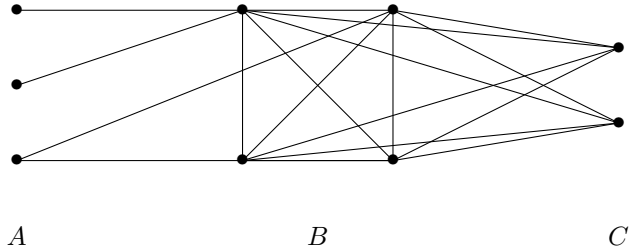

This result was improved by Sliva in 2012 \cite{sliva2012}:

\begin{thm}\label{slivaed}
   Let $G = (V,E)$ be a graph. $G$ is antimagic if $G$ admits a partition $A,B,C$ of $V$ such that:
    
    \begin{itemize}
        \item $B$ induces a connected $r$-regular subgraph, with $2 \leq r \leq |B| - 1$.
        \item For every vertex $v \in A$, $N(v) \subset B$.
        \item For every vertex $v \in C$, $B \subset N(v)$.
        \item If $r < |B| - 1$, then $C$ has $\alpha$ vertices, with $\alpha \geq \frac{2|B|^2 + |B|}{r} - 2 |B| + 2$ and $|E(A,B)| \leq \frac{\alpha}{2}$
    \end{itemize}
\end{thm}

Finally, the following result was proved by Yilma in 2013 \cite{yilma2013}; note that it is similar to Theorem \ref{barrused} with $|B| = 2$.

\begin{thm}
    \label{yilmaed}
    Let $G = (V,E)$ be a graph. If $G$ has a vertex $x$ such that $d(x) = \Delta(G) \geq \frac{2n}{3}$, and $x$ has a neighbor $y$ such that $\{x,y\}$ is a dominating clique in $G$, then $G$ is antimagic.
\end{thm}

    In the case $|B| \geq 4$, we improve Theorem \ref{barrused} with the following result:

\begin{restatable}{thm}{dominanted}\label{dominanted}
    Let $G = (V,E)$ be a graph with a dominating clique $K$ of order $|K| \geq 4$. If, for every vertex $v \in V \setminus K$, $d(v) \leq \min\limits_{u \in K} d(u)$, then $G$ is antimagic.
\end{restatable}

Indeed, using the notations from Figure \ref{barrusillus}, we can easily see that every vertex $v \in A$ is such that $d(v) \leq d(u)$ for any $u \in B$. Moreover, since any $u \in B$ is incident to every vertex $v \in C$, we also obtain $d(v) \leq d(u)$.

We also introduce the following definition (recall that $\sigma(u)$ is the sum of labels over edges incident to $u$):

\begin{definition}
    We say that a given graph $G = (V,E)$ admits a \emph{$C$-antimagic injection} if there exists an injective labelling $f : E \rightarrow \{1,2,\ldots,m + C\}$ such that the values of the $\sigma(u)$, for $u \in V$, are pairwise distinct.
\end{definition}

If such a labelling exists, we say that $G$ is $C$-antimagic, for some $C \geq 0$. Note that, if $C = 0$, the labelling (as well as the graph) is antimagic.

This allows us to show the following result:

\begin{restatable}{thm}{injectioned}\label{injectioned}
    Let $G = (V,E)$ be a graph with a dominating clique $K$ of order $k = |K| \geq 3$. If $m \geq 3(n-k) - 2 + \frac{k(k-1)}{2}$, then $G$ is $3$-antimagic.
\end{restatable}

Recall that the number of edges in a clique $K$ of order $k$ is $\frac{k(k-1)}{2}$.

In the rest of the article, at any point when constructing a labelling, we will denote by $\sigma'(u)$ the partial sum of edges incident to $u$ that have already been labelled. Once every edge incident to $u$ has been labelled, $\sigma'(u) = \sigma(u)$.

\section{Proof of Theorem \ref{dominanted}}

\begin{proof}
    Let $G = (V,E)$ be a graph with a dominating clique $K = \{u_i, 1 \leq i \leq k\}$ with $k \geq 4$, such that, for every vertex $v \in V \setminus K$, $d(v) \leq \min\limits_{u \in K} d(u)$. Let $V \setminus K = \{v_i, 1 \leq i \leq n - k\}$.

    We start by partitioning the set of edges $E = E_1 \cup E_2 \cup E_3 \cup E_4$ and we will label them in this order. The sets are defined in the following way:

    \begin{itemize}
        \item $E_1$ is the set of edges between two vertices of $V \setminus K$.
        \item $E_3$ is a set of edges such that every vertex $v \in V \setminus K$ has exactly one incident edge in $E_3$ (towards a vertex of $K$, otherwise the edge would be in $E_1$). Note that this is always possible since $K$ is a dominating set in $G$; also note that $|E_3| = n - k$.
        \item $E_4$ is the set of edges between two vertices of $K$.

        \item $E_2$ is the set of the remaining edges in $E$. Note that every edge in $E_2$ has one endpoint in $K$ and one endpoint in $V \setminus K$.
    \end{itemize}

    We now construct an explicit labelling of $E$ with the following steps:

    \begin{itemize}
        \item[Step $1$:] We label the edges of $E_1$, arbitrarily assigning labels in $[1;|E_1|]$.

        \item[Step $2$:] We label the edges of $E_2$, picking labels in $[|E_1| + 1; |E_1| + |E_2|]$ using an algorithm we will describe later on.

        \item[Step $3$:] We sort the vertices of $V \setminus K$ by increasing order of their value of $\sigma'$: $\sigma'(v_1) \leq \sigma'(v_2) \leq \ldots \leq \sigma'(v_{n-k})$. We then label the edges of $E_3$ following this order - recall that there is exactly one edge incident to each $v_i$ in $E_3$ - , picking labels in increasing order in $[|E_1| + |E_2| + 1; |E_1| + |E_2| + |E_3|]$.

        \item[Step $4$:] We sort the vertices of $K$ by increasing order of their value of $\sigma'$: $\sigma'(u_1) \leq \sigma'(u_2) \leq \ldots \leq \sigma'(u_k)$. We then label edges of $E_4$ by visiting them in lexicographic order, and picking labels in increasing order in $[|E_1|+ |E_2| + |E_3| + 1; |E_1| + |E_2| + |E_3| + |E_4|]$. Hence, $u_1u_2$ will be labelled by $|E_1|+|E_2|+|E_3|+1$, $u_1u_3$ with $|E_1|+|E_2|+|E_3|+2$, and so on.
    \end{itemize}

    For any two vertices $v_i,v_j \in V \setminus K$, with $i < j$, notice that Step $3$ allows us to assert that, once the labelling is done, $\sigma(v_i) < \sigma(v_j)$. Similarly, for any two vertices $u_i,u_j \in K$, with $i < j$, Step $4$ allows us to assert that, once the labelling is done, $\sigma(u_i) < \sigma(u_j)$.
    
    In order to obtain an antimagic labelling, we only need to assign the labels during Step $2$ in order to guarantee that, at the end of the labelling, for any two vertices $u \in K$ and $v \in V \setminus K$, $\sigma(u) \neq \sigma(v)$. We will actually prove that $\sigma(u) > \sigma(v)$ by assigning labels along \emph{maximal paths} made of edges in $E_2$, and alternating in those paths between the smallest label available and the largest one.

    We now wish to formally describe the algorithm used to label the edges of $E_2$. Suppose that $E_2$ is not empty (otherwise there is nothing to do during this step) and let $\alpha = \min\limits_{u \in K}d(u)$. We have $\alpha = k - 1 + \beta$ for some $\beta \geq 0$. It was shown in \cite{alon2004} by Alon et al. that graphs $G$ with $\Delta(G) \geq n - 2$ are antimagic. We will then suppose in the following that $\alpha \leq n - 3$.

    To label the edges of $E_2$, we will visit them by creating \emph{maximal} paths; this means that every time the current path can be extended by picking and labelling another edge in $E_2$, we will do so.

    We start from a vertex in $K$. Recall that every edge in $E_2$ has one endpoint in $K$ and one endpoint in $V \setminus K$; this allows us to build a path alternating between vertices in $K$ (when an even number of edges have been labelled in $E_2$) and vertices in $V \setminus K$ (when an odd number of edges have been labelled in $E_2$). When we reach a vertex $u \in K$, if $u$ has a neighbor $v$ such that $uv \in E_2$ and $uv$ has not yet been labelled, we extend the current path by labelling $uv$. Otherwise, we arbitrarily pick another vertex $u' \in K$ such that $u'$ has an incident edge $u'v \in E_2$ not yet labelled, and we start another path by labelling $u'v$.

    Similarly, every time we reach a vertex $v \in V \setminus K$: if $v$ has an incident edge $vu \in E_2$ not yet labelled, we extend the current path by labelling $vu$, otherwise we start a new path by arbitrarily picking a vertex $v' \in V \setminus K$ with an incident edge $v'u \in E_2$ not yet labelled, and we label $v'u$.

    Every time we go from a vertex in $K$ to a vertex in $V \setminus K$, we will use the largest label available, and every time we go from a vertex in $V \setminus K$ to a vertex in $K$, we will use the smallest label available. Alternating between large labels and small labels will allow us to obtain upper bounds on the values of $\sigma(v)$, for every $v \in V \setminus K$, and guarantee that they are lower than the values of $\sigma(u)$, for every $u \in K$.

    Formally, we begin Step $2$ by arbitrarily picking some vertex $u \in K$, such that $u$ has at least one incident edge $uv \in E_2$. We label $uv$ with the largest label $|E_1| + |E_2|$.

    Then, when we reach a vertex $v \in V\setminus K$ once $2\gamma + 1$ edges have been labelled in $E_2$, for some $\gamma \geq 0$: if $v$ has an incident edge $uv \in E_2$ not yet labelled, we extend the current path by labelling it with $|E_1| + \gamma + 1$. Otherwise, every edge incident to $v$ in $E_2$ has been labelled, and we arbitrarily pick another vertex $v' \in V\setminus K$ such that $v'$ has an incident edge $v'u \in E_2$ not yet labelled (if such a $v'$ does not exist, the labelling of $E_2$ is done). We then label $v'u$ with $|E_1| + \gamma + 1$.

    Similarly, when we reach a vertex $u \in K$ once $2\gamma$ edges have been labelled in $E_2$, for some $\gamma \geq 1$: if $u$ has an incident edge $uv \in E_2$ not yet labelled, we label $uv$ with $|E_1| + |E_2|- \gamma$. If this is not the case, we arbitrarily pick another vertex $u' \in K$ such that $u'$ has an incident edge $u'v \in E_2$ not yet labelled, and we label $u'v$ with $|E_1| + |E_2| - \gamma$.

    ~\hfill

    Let $v \in V \setminus K$. Every edge incident to $v$ in $E_2$ (recall that there are at most $\alpha - 1$ such edges) was labelled during one of the three following cases:

    \begin{itemize}
        \item $v$ is an interior vertex of one of the maximal paths; in this case, two edges incident to $v$ are labelled consecutively, and their sum is equal to $(|E_1| + \gamma + 1) + (|E_1| + |E_2| - \gamma) = 2|E_1| + |E_2| + 1$.

        \item $v$ is the starting vertex of a maximal path, with $vw$ the first edge of this path; then $vw$ is labelled with a label $f(vw) \leq \frac{1}{2}(|E_1| + 1 + |E_1| + |E_2|)$ since we always use the smallest label available when going from a vertex in $V \setminus K$ to a vertex in $K$.

        \item $v$ is the endpoint of a maximal path, with $wv$ the last edge of the path. This means it was not possible to extend the path by finding another unlabelled edge in $E_2$ incident to $v$, and it can only happen at most once. In this case, $wv$ is labelled with a label $f(wv) \leq |E_1| + |E_2|$.
    \end{itemize}

    Overall, after Step $2$, we obtain $\sigma'(v) \leq \frac{1}{2}(2|E_1| + |E_2| + 1)(\alpha - 2) + |E_1| + |E_2| = (\alpha - 1)|E_1| + \frac{\alpha}{2}|E_2| + \frac{\alpha-2}{2}$.

    Similarly, for every vertex $u \in K$, every edge incident to $u$ in $E_2$ was labelled during one of the three following cases:

    \begin{itemize}
        \item $u$ is an interior vertex of one of the maximal paths; in this case, two edges incident to $u$ are labelled consecutively, and their sum is equal to $(|E_1| + (\gamma - 1) + 1) + (|E_1| + |E_2| - \gamma) = 2|E_1| + |E_2|$.

        \item $u$ is the starting vertex of a maximal path, with $uw$ the first edge of this path; then, $uw$ is labelled with a label $f(uw) \geq \frac{1}{2}(|E_1| + 1 + |E_1| + |E_2|)$ since we always use the largest label available when going from a vertex in $K$ to a vertex in $V \setminus K$.

        \item $u$ is the endpoint of a maximal path, with $uw$ the last edge of the path. This can only happen once, and $uw$ was labelled with some $f(uw) \geq |E_1| + 1$.
    \end{itemize}

    Overall, after Step $3$, we obtain $\sigma'(u) \geq \frac{2|E_1| + |E_2|}{2}(\beta - 1) + |E_1|+ 1 = \beta|E_1| + \frac{\beta - 1}{2}|E_2| + 1$. Recall that edges incident to $u$ in $E_3$, labelled during Step $3$, are all labelled with a label greater than $|E_1| + |E_2| + 1$.

    Let $v \in V \setminus K$. Recall that the remaining edge incident to $v$ is labelled during Step $3$ (and that there is only one left to label once Steps $1$ and $2$ are done), and that once Step $3$ is done, all edges incident to $v$ will have been labelled. Then, we obtain, for every $v \in V \setminus K$ and every $u \in K$:

    \begin{align*}
        \sigma(v) &\leq (|E_1| + |E_2| + |E_3|) + (\alpha-1)|E_1| + \frac{\alpha}{2}|E_2| + \frac{\alpha - 2}{2} \\
        \sigma(u) &\geq (k-1)(|E_1| + |E_2| + |E_3|) + \frac{k(k-1)}{2} + \beta|E_1| + \frac{\beta - 1}{2}|E_2| + 1
    \end{align*}

    where the $(k-1)(|E_1| + |E_2| + |E_3|) + \frac{k(k-1)}{2}$ term is the minimal sum of labels of edges incident to $u$ in $E_4$ (since for every vertex $u \in K$, there are $k-1$ edges incident to $u$ in $E_4$, and edges in $E_4$ are the last to be labelled).

    Let $\Gamma = \sigma(u) - \sigma(v)$. We obtain, recalling that $\beta - 1 - \alpha = -k$:

    \begin{align*}
        \Gamma &\geq (k-2)(|E_1| + |E_2| + |E_3|) + \frac{k(k-1)}{2} + 1 + |E_1|(2 - k) - \frac{k}{2}|E_2| - \frac{\alpha - 2}{2} \\
        & \geq \left(\frac{k}{2} - 2\right)|E_2| + (k-2)|E_3| + \frac{k(k-1)}{2} + 1 - \frac{\alpha - 2}{2}
    \end{align*}

    We wish to show that $\Gamma > 0$. Recall that $|E_3| = n - k$, $\alpha \leq n - 3$, and $k \geq 4$. This yields $\left(\frac{k}{2} - 2\right)|E_2| \geq 0$ and $(k-2)|E_3| - \frac{\alpha-2}{2} \geq 2|E_3| - \frac{\alpha - 2}{2} = \frac{4(n-k) - \alpha + 2}{2}$.

    We obtain $\Gamma \geq \frac{1}{2}(4(n-k) + 2 - \alpha) + \frac{k(k-1)}{2}+ 1 \geq \frac{k(k-1)}{2} - \frac{\alpha - 2}{2} + 2(n-k)$.

    If $4(n-k) + 2 \geq \alpha$, we can immediately conclude that $\Gamma \geq \frac{k(k-1)}{2} + 1 > 0$. Suppose then that $4(n-k) + 2 \leq \alpha \leq n - 3$. We have $k \geq \frac{3}{4}n + \frac{5}{4}$, meaning $k \geq \frac{3}{4}n$. In this case, we obtain, since $\alpha \leq n - 3$:
    
    \begin{align*}
    \frac{k(k-1)}{2} - \frac{\alpha - 2}{2} &\geq \frac{3n(\frac{3}{4}n - 1)}{8} - \frac{n - 5}{2} \\
    & \geq \frac{1}{8}\left(\frac{9}{4}n^2-7n+20\right) > 0
    \end{align*}

    Therefore, $\Gamma > 0$ in every case, meaning that, for every $u \in K$ and $v \in V \setminus K$, we always obtain $\sigma(u) > \sigma(v)$. This, in turn, means that the labelling obtained is antimagic and that $G$ is antimagic.
    
\end{proof}

\section{Proof of Theorem \ref{injectioned}}

We recall Theorem \ref{injectioned}:

\injectioned*

\begin{proof}

    Let $K = \{u_i, 1 \leq i \leq k\}$ and $V \setminus K = \{v_i, 1 \leq i \leq n - k\}$. We start by partitioning the set of edges to label $E = E_1 \cup E_2 \cup E_3$, in the following way:

    \begin{itemize}
        \item $E_3$ is the set of edges between two vertices of $K$.
        \item $E_2$ is a set of edges such that every vertex $u \in V \setminus K$ has exactly one incident edge in $E_2$ (towards a vertex of $K$). This is always possible since $K$ is a dominating clique in $G$.
        \item $E_1$ is the set of the remaining edges.
    \end{itemize}

    We will proceed similarly as in the proof of Theorem \ref{dominanted}, but with an additional idea: we want to create some \lq space\rq \, between the values of $\sigma(u), u \in V$, once the labelling is done, in order to be able to switch some labels to resolve some hypothetical collisions between some $\sigma(u)$. In order to do so, we start by defining the set of labels for the three sets $E_1,E_2,E_3$:

    \begin{itemize}
        \item Edges in $E_2$ will be assigned labels in $L_2 = \{1,4,7,\ldots,1 + 3(n-k-1)\}$ - recall that there are exactly $n-k$ edges in $E_2$.

        \item Edges in $E_3$ will be assigned labels in $L_3 = \{m - \frac{k(k-1)}{2} + 1, m - \frac{k(k-1)}{2}+2,\ldots,m\}$, meaning the $\frac{k(k-1)}{2}$ largest labels. Notice that since $m - \frac{k(k-1)}{2} \geq 3(n-k)-2$, the smallest label assigned to $E_3$ is strictly higher than the largest label assigned to $E_2$.

        \item Edges in $E_1$ will be assigned the remaining labels in $L_1 = \{1,2,\ldots,m\} \setminus (L_2 \cup L_3)$.
    \end{itemize}

    The labelling of $G$ consists of the following steps:

    \begin{itemize}
        \item[Step $1$:] We arbitrarily label edges in $E_1$ with labels in $L_1$.

        \item[Step $2$:] We sort the vertices of $V \setminus K$ in increasing order of their value of $\sigma'$: $\sigma'(v_1) \leq \sigma'(v_2) \leq \ldots \leq \sigma'(v_{n-k})$. We then label edges of $E_2$ following this order - recall that there is exactly one incident edge in $E_2$ to each $v_i$ -, picking labels in $L_2$ in increasing order.

        \item[Step $3$:] We sort the vertices of $K$ by increasing order of their value of $\sigma'$: $\sigma'(u_1) \leq \sigma'(u_2) \leq \ldots \leq \sigma'(u_k)$. We then label edges in $E_3$ by visiting them in lexicographic order, and picking labels in $L_3$ in increasing order. 
    \end{itemize}

    We state some properties that directly result from the description of Steps $2$ and $3$, and the definition of $L_2$:

    \begin{prop}\label{ecart3}
    For any two vertices $v_i,v_j \in V \setminus K$, with $i \neq j$, $|\sigma(v_i) - \sigma(v_j)| \geq 3$.
    \end{prop}

    \begin{prop}\label{ecartclique}
    For any vertex $u_i \in K$, with $i \leq k - 1$, $\sigma(u_{i+1}) - \sigma(u_i) \geq k - 2$.
    \end{prop}

    In all the following, we will call \emph{conflict} an equality $\sigma(x) = \sigma(y)$, for two distinct vertices $x$ and $y$.

    Let us suppose for now that $k \geq 5$; we will explain later on how to deal with the cases where $k = 3$ and $ k = 4$. Note that the two post-processing steps described below take place after Step $3$ of the labelling has been completed. All edges are now labelled, and we want to check if the labelling is antimagic, and resolve the conflicts if there are any.

    In the following, we will denote by $\stild(u)$ the value of $\sigma(u)$ before starting this Post-processing $1$. To facilitate the understanding of the proof, we will still use $\sigma(u)$ to denote at any point the sum of labelled edges incident to $u$.

    \begin{itemize}

        \item Post-processing $1$: 

         The only hypothetical conflicts can happen between some $u_i \in K$ and some $v_j \in V \setminus K$, according to the properties we just showed. Let us consider the $u_i$ in increasing order of $i$, for $1 \leq i \leq k - 2$. If, when we reach some vertex $u_i$, there exists another vertex $v_j \in V \setminus K$ such that $\sigma(u_i) = \sigma(v_j)$, we switch the label of $u_iu_k$ and the label of $u_{i+1}u_{i+2}$. Note that these are well defined since we supposed $1 \leq i \leq k - 2$.

    Since we labelled the edges of $K$ by sorting them in lexicographic order, these two edges have consecutive labels. By switching the two labels, we are then increasing the value of $\sigma(u_i)$ and $\sigma(u_k)$ by $1$, and decreasing $\sigma(u_{i+1})$ and $\sigma(u_{i+2})$ by $1$. Thanks to Property \ref{ecart3}, we know that by increasing $\sigma(u_i)$ by $1$ (without changing any values of $\sigma(v)$, for any $v \in V \setminus K$), we are guaranteeing that $\sigma(u_i)$ cannot be equal to any other value of $\sigma(v)$, for $v \in V \setminus K$. We will now explain why $u_i$ cannot be in conflict with another $u_j \in K$.

    Note that the only other changes that could happen due to the switching of some labels are the values of $\sigma(u_{i+1}), \sigma(u_{i+2})$ and $\sigma(u_k)$; however, since the $u_i$ are treated in increasing order, we will deal with those vertices later on, solving any potential conflict that could happen between those vertices and other vertices.

    Once all the $u_i$ have been treated, with $1 \leq i \leq k - 2$, we have guaranteed that, for any $1 \leq i \leq k - 2, 1 \leq j \leq n - k$, $\sigma(u_i) \neq \sigma(v_j)$. Moreover, for $i \in [1;k-2]$, the value of $\stild(u_i)$ might have been modified:

    \begin{itemize}
        \item by $+1$ when we reached $u_i$
        \item by $-1$ when we reached $u_{i-1}$ (if it exists)
        \item by $-1$ when we reached $u_{i-2}$ (if it exists)
    \end{itemize}

    Note that reaching another vertex (say $u_{i-1}$ for instance) does not automatically mean we modified $\stild(u_i)$ by $-1$; this only happens if there was a conflict involving $u_{i-1}$ when we reached it.

    Moreover, if the value of $\stild(u_{i+1})$ was decreased by $2$, this means the value of $\stild(u_{i})$ was also decreased by $1$ (when we reached $u_{i-1}$). In particular, since $k \geq 5$, this means that we still have $\sigma(u_{i+1}) - \sigma(u_i) \geq (k-2) - 2 > 0$, hence $\sigma(u_{i+1}) > \sigma(u_i)$ after treating all the vertices $u_p, 1 \leq p \leq k - 2$, thanks to Property \ref{ecartclique}.

    \item Post-processing $2$:

    We have to deal with conflicts involving $u_{k-1}$ and $u_k$. If there are no conflicts, then the labelling is antimagic and we are done. Otherwise, we will replace the label of $u_{k-1}u_k$ by a label picked among $\{m+1,m+2,m+3\}$. Let us denote $\lambda_1 = \sigma(u_{k-1})$ and $\mu_1 = \sigma(u_k)$ after labelling $u_{k-1}u_k$ with $m+1$. Note that, for any $1 \leq i \leq k - 2$, we still have $\sigma(u_i) < \lambda_1 < \mu_1$ (note that these inequalities stay valid if we pick $m+2$ or $m+3$ instead of $m+1$).

    Due to Property \ref{ecart3}, there exists at most one vertex $v_j \in V \setminus K$ such that $\sigma(v_j) \in \{\lambda_1,\lambda_1 + 1, \lambda_1 + 2\}$, which correspond to the different possible values of $\sigma(u_{k-1})$ depending on the value of $f(u_{k-1}u_k)$. Similarly, there exists at most one vertex $v_{j'} \in V \setminus K$ such that $\sigma(v_{j'}) \in \{\mu_1,\mu_1+1, \mu_1+2\}$. Overall, this means that at least one label among $\{m+1,m+2,m+3\}$ will guarantee that there are no conflicts involving $u_{k-1}$ or $u_k$, in turn guaranteeing that the labelling described is $3$-antimagic (meaning $G$ is $3$-antimagic).

    \end{itemize}

    We complete the proof with the cases $k = 3$ and $k = 4$:

    \begin{claim}
        If $k = 3$, $G$ is $3$-antimagic.
    \end{claim}

    \begin{proof}
        Once Steps $1,2,$ and $3$ are done, we obtain: $f(u_1u_2) = m-2, f(u_1u_3) = m - 1$ and $f(u_2u_3) = m$.

        If there exists a vertex $v \in V \setminus K$ such that $\sigma(v) = \sigma(u_1)$, we replace $f(u_1u_3)$ by $m$ and $f(u_2u_3)$ by $m+1$. We still have $\sigma(u_1) < \sigma(u_2) < \sigma(u_3)$, and since we increased $\sigma(u_1)$ by $1$, we have guaranteed there are no conflicts involving $u_1$ thanks to Property \ref{ecart3}.

        Finally, similarly to the proof with $k \geq 5$, we pick a label for $f(u_2u_3)$ among $\{m+1,m+2,m+3\}$ in order to guarantee that there are no conflicts involving $u_2$ or $u_3$.
    \end{proof}

    \begin{claim}
        If $k = 4$, $G$ is $3$-antimagic.
    \end{claim}
    \begin{proof}
        We apply the same algorithm as in the case $k \geq 5$, however Property \ref{ecartclique} does not allow us to automatically claim that $\sigma(u_i) \neq \sigma(u_j)$ for any two vertices $u_i,u_j \in K$ since $k - 2 = 2$.

        We have $\stild(u_1) \leq \stild(u_2) - 4$. We then necessarily still have $\sigma(u_1) < \sigma(u_2)$ after the two post-processing steps.

        We have $\stild(u_2) \leq \stild(u_3) - 2$, due to Property \ref{ecartclique}. We also have $\sigma(u_2) \leq \stild(u_2) + 1$ (since $\sigma(u_2)$ can only be modified during Post-processing $1$), and since $f(u_3u_4) \in \{m+1,m+2,m+3\}$, $\sigma(u_3) \geq \stild(u_3)$. Indeed, if $\stild(u_3)$ has been decreased by $1$ when we reached $u_1$ and also when we reached $u_2$, then before starting Post-processing $2$, we have $f(u_3u_4) = m - 1$. Since $f(u_3u_4)$ is then modified to a label in $\{m+1,m+2,m+3\}$ during Post-processing $2$, we obtain $\sigma(u_3) \geq \stild(u_3)$. Overall we obtain $\sigma(u_2) < \sigma(u_3)$.

        Finally, before starting Post-processing $2$ and modifying the label of $u_3u_4$, we necessarily have $\sigma(u_3) < \sigma(u_4)$ at this point. Indeed, $\sigma(u_4) \geq \stild(u_4)$ since every switching of labels increases the value of $\sigma(u_4)$ by $1$, and $\sigma(u_3) \leq \stild(u_3)$ since there were only two possible switchings impacting $\sigma(u_3)$, both decreasing its value by $1$.

        We then have $\sigma(u_3) < \sigma(u_4)$, and the labelling is $3$-antimagic.
    \end{proof}

    We then have two specific proofs of $3$-antimagicness for $k = 3$ and $k = 4$, and a proof for $k \geq 5$, overall proving the theorem. 
\end{proof}

\section{Conclusion and future work}

In this paper, we study the antimagicness of graphs with a dominating clique, improving on the work of \cite{barrus2010}. We also introduce the definition of $C$-antimagicness, which can intuitively be seen as a form of weak antimagicness. This leads to the following conjecture:

\begin{conj}
    \label{conjpascentrale}
    Every connected graph (except $K_2$) is $C$-antimagic, for some constant $C \geq 0$.
\end{conj}

Notice that $C = 0$ is Conjecture \ref{conjcentrale}; however there might exist some `easy' constructions yielding a proof of $C$-antimagicness over some classes of graphs and allowing incremental progression towards a proof of Conjecture \ref{conjcentrale}.

Some antimagicness results over graphs with a dominating clique are also still open; for instance, graphs with a dominating $K_2$ such that $\frac{n}{2} \leq \Delta(G) \leq \frac{2n}{3}$ (when $\Delta(G) \geq \frac{2n}{3}$, Theorem \ref{yilmaed} applies) are not shown yet to be antimagic.

It might also be possible to adapt the proofs given in this article, to show antimagicness results over other similar classes of graphs. Recall that graphs $G$ such that $\Delta(G) \geq n-3$ are antimagic (\cite{alon2004}, \cite{yilma2013}). We think that similar ideas can be used to show that connected graphs $G$ such that $\Delta(G) = n - 4$ are antimagic, provided that their number of edges $m$ is large enough, \emph{i.e.} $m \geq cn$ for some constant $c > 1$. It is also interesting to study the case where $c$ is not a constant; for instance:

\begin{question}
    Are graphs $G$ such that $\Delta(G) \geq n - k$ for some $k$ and $m \geq f(k)\cdot n$, for some positive function $f$, antimagic?
\end{question}

\bibliographystyle{elsarticle-num}
\bibliography{biblio}

\end{document}